\documentclass[final]{siamltex}
\usepackage{amsfonts, amssymb, amsmath, graphicx, dcolumn, bm, slashed, fancyhdr}
\newcommand{\bd}{\mathbf}
\newcommand{\beq}{\begin{eqnarray}}
\newcommand{\eeq}{\end{eqnarray}}

\title{Calculating determinants of block matrices}

\author{Philip D. Powell\footnotemark[2]}

\begin{document}
\maketitle

\footnotetext[2]{Department of Physics, University of Illinois at Urbana-Champaign, 1110 W. Green St., Urbana, IL 61801 (ppowell2@illinois.edu).}

\begin{abstract}
This paper presents a method for expressing the determinant of an $N \times N$ complex block matrix in terms of its constituent blocks.  The result allows one to reduce the determinant of a matrix with $N^2$ blocks to the product of the determinants of $N$ distinct combinations of single blocks.  This procedure proves useful in the analytic description of physical systems with multiple discrete variables, as it provides a systematic method for evaluating determinants which might otherwise be analytically intractable.
\end{abstract}



\section{Introduction}
Block matrices are ubiquitous in physics and applied mathematics, appearing naturally in the description of systems with multiple discrete variables (e.g., quantum spin, quark color and flavor)~\cite{Abuki,Rossner,Sasaki}.  In addition, block matrices are exploited in many computational methods familiar to researchers of fluid dynamics~\cite{Cliffe, Murman}.  The need to calculate determinants of these matrices is almost equally widespread, for both analytical and numerical applications~\cite{Molinari,Popescu}.  For example, a model of high density quark matter must include color (3), flavor (2-6), and Dirac (4) indices, giving rise to a matrix between size $24 \times 24$ and $72 \times 72$, with any additional properties enlarging the matrix further~\cite{Abuki,Powell}.

The problem of calculating the determinant of a $2 \times 2$ block matrix has been long studied, and is a most important case, since it can be extended to any larger matrix in the same way that the determinant of an arbitrary square matrix can be expressed in terms of the determinants of $2 \times 2$ matrices, via minor expansion~\cite{Silvester}.  The solution to this problem can be obtained by considering the equation
\beq
\begin{bmatrix}
\bd{A} & \bd{B} \\
\bd{C} & \bd{D}
\end{bmatrix}
\begin{bmatrix}
\bd{I} & \bd{0} \\
-\bd{D}^{-1} \bd{C} & \bd{I}
\end{bmatrix} = 
\begin{bmatrix}
\bd{A - BD}^{-1} \bd{C} & \bd{B} \\
\bd{0} & \bd{D}
\end{bmatrix}   ,   \label{eq:2x2block}
\eeq
where $\bd{A}$, $\bd{B}$, $\bd{C}$, and $\bd{D}$ are $k \times k$, $k \times (N-k)$, $(N-k) \times k$, and $(N-k) \times (N-k)$ matrices respectively, $\bd{I}$ and $\bd{0}$ are the identity and zero matrix respectively (taken to be of the appropriate dimension), and we have assumed that $\bd{D}$ is invertible.  If $\bd{D}^{-1}$ does not exist, Eq. (\ref{eq:2x2block}) can be rewritten in terms of $\bd{A}^{-1}$, with the right side lower-triangular~\cite{Tian}.  If neither inverse exists, the notion of generalized inverses must be employed~\cite{Moore,Penrose,Meyer}.  We do not consider these complications in the present analysis, but rather simply assume the existence of all inverses that arise in the following calculations.  In practice, this assumption is of nearly no consequence, since the blocks will generally be functions of at least one continuous variable (e.g., momentum), and any singularities will be isolated points of measure zero.

Taking the determinant of Eq. (\ref{eq:2x2block}) yields the result~\cite{Silvester}
\beq
\det \left( \begin{matrix}
\bd{A} & \bd{B} \\
\bd{C} & \bd{D}
\end{matrix} \right) = \det \left( \bd{A - BD}^{-1} \bd{C} \right) \det \left( \bd{D} \right)   ,   \label{eq:2x2det}
\eeq
where we have exploited the fact that the determinant of a block tringular matrix is the product of the determinants of its diagonal blocks.  The matrices appearing on the right side of Eq. (\ref{eq:2x2det}) are the $(N-k) \times (N-k)$ matrix $\bd{D}$, and the $k \times k$ Sch\"{u}r complement with respect to $\bd{D}$~\cite{Ouellette,Zhang}.

While Eq. (\ref{eq:2x2det}) has proven immensely useful, there remain many applications in which this result is not ideally suited.  For example, there are situations in which physics dictates that we partition a matrix in a form different than in Eq. (\ref{eq:2x2block}).  The description of intermediate-energy quark matter in quantum chromodynamics, for instance, involves propagators which possess four Dirac indices, three color indices and three flavor indices.  While one could, in principle, partition the resulting $36 \times 36$ matrix as shown in Eq. (\ref{eq:2x2det}), such a partitioning would treat the Dirac indices asymmetrically, eliminating two of the indices while retaining two others.  While such a division is of no consequence in computational applications, it is quite undesirable in analytical descriptions.  Thus, we are led to desire a generalization of Eq. (\ref{eq:2x2det}) which holds for a block matrix of more general partitioning.

The remainder of this paper is organized as follows.  In Sec.~\ref{sec:directcalc}, a direct, ``brute force" calculation of the determinant of an $N \times N$ complex block matrix is presented, which manifests clearly the emergence of a Sch\"{u}r complement structure.  In Sec.~\ref{sec:induction}, an alternate proof of our result is given, which is somewhat more elegant, but perhaps less intuitive, by means of induction on $N$.  Finally, in Sec.~\ref{sec:examples}, some applications of our result are discussed, including the cases $N = 2$ and $N = 3$, as well as a $48 \times 48$ matrix with $N = 6$, which arises in the study of quark matter.

\section{Direct Calculation \label{sec:directcalc}}

\hspace{0mm}

\begin{theorem}
Let $\bd{S}$ be an $(nN) \times (nN)$ complex matrix, which is partitioned into $N^2$ blocks, each of size $n \times n$:
\beq
\bd{S} = \begin{bmatrix}
		\bd{S}_{11} & \bd{S}_{12} & \cdots & \bd{S}_{1N} \\
		\bd{S}_{21} & \bd{S}_{22} & \cdots & \bd{S}_{2N} \\
		\vdots  & \vdots  & \ddots & \vdots  \\
		\bd{S}_{N1} & \bd{S}_{N2} & \cdots & \bd{S}_{NN}
	 \end{bmatrix}   .   \label{eq:Smatrix}
\eeq
Then the determinant of $\bd{S}$ is given by
\beq
\det(\bd{S}) = \prod^N_{k=1} \det (\bm{\alpha}^{(N-k)}_{kk})   ,
\eeq
where the $\alpha^{(k)}$ are defined by
\beq
\bm{\alpha}^{(0)}_{ij} & = & \bd{S}_{ij}   \nonumber   \\
\bm{\alpha}^{(k)}_{ij} & = & \bd{S}_{ij} - \bm{\sigma}^T_{i,N-k+1} \tilde{\bd{S}}^{-1}_{k} \bd{s}_{N-k+1,j}   , \hspace{10mm}   k \geq 1,   \label{eq:alphas}
\eeq
and the vectors $\bm{\sigma}^T_{ij}$ and $\bd{s}_{ij}$ are
\beq
\bd{s}_{ij} = \left(\begin{array}{cccc} \bd{S}_{ij} & \bd{S}_{i+1,j} & \cdots & \bd{S}_{Nj} \end{array} \right)^T   , \hspace{3mm} 
\bm{\sigma}^T_{ij} = \left(\begin{array}{cccc} \bd{S}_{ij} & \bd{S}_{i,j+1} & \cdots & \bd{S}_{iN} \end{array} \right)   .   \label{eq:vecs}
\eeq

\end{theorem}

\begin{proof}
In order to facilitate computing the determinant of $\bd{S}$, we define a lower triangular auxiliary matrix
\beq
\bd{U} = \begin{bmatrix}
		\bd{I} & \bd{0} & \cdots & \bd{0} \\
		\bd{U}_{21} & \bd{I} & \cdots & \bd{0} \\
		\vdots & \vdots & \ddots & \vdots \\
		\bd{U}_{N1} & \bd{U}_{N2} & \cdots & \bd{I}
	 \end{bmatrix}   .
\eeq
Forming the product $\bd{SU}$ gives
\beq
\bd{SU} = \begin{bmatrix}
		\bd{S}_{11} + \bd{S}_{12} \bd{U}_{21} + \cdots \bd{S}_{1N} \bd{U}_{N1} & \bd{S}_{12} + \bd{S}_{13} \bd{U}_{32} + \cdots \bd{S}_{1N} \bd{U}_{N2} & \cdots & \bd{S}_{1N} \\
		\bd{S}_{21} + \bd{S}_{22} \bd{U}_{21} + \cdots \bd{S}_{2N} \bd{U}_{N1} & \bd{S}_{22} + \bd{S}_{23} \bd{U}_{32} + \cdots \bd{S}_{2N} \bd{U}_{N2} & \cdots & \bd{S}_{2N} \\
		\vdots & \vdots & \ddots & \vdots \\
		\bd{S}_{N1} + \bd{S}_{N2} \bd{U}_{21} + \cdots \bd{S}_{NN} \bd{U}_{N1} & \bd{S}_{N2} + \bd{S}_{N3} \bd{U}_{32} + \cdots \bd{S}_{NN} \bd{U}_{N2} & \cdots & \bd{S}_{NN}
	  \end{bmatrix}   .   \nonumber \\   \label{eq:SUNxN}
\eeq
Now, let us assume that the blocks of $\bd{U}$ can be chosen such that the product $\bd{SU}$ is upper triangular.  In this case, we obtain $(N^2 - N)/2$ constraint equations (the lower triangular blocks of $\bd{SU} = \bd{0}$) with $(N^2 - N)/2$ unknowns (the blocks of $\bd{U}$).  The $N - k$ equations arising from the $k$th column of $\bd{SU}$ are
\beq
\bd{S}_{k+1,k} + \bd{S}_{k+1,k+1} \bd{U}_{k+1,k} + \cdots + \bd{S}_{k+1,N} \bd{U}_{Nk} & = & \bd{0}   ,   \nonumber  \\
\bd{S}_{k+2,k} + \bd{S}_{k+2,k+1} \bd{U}_{k+1,k} + \cdots + \bd{S}_{k+2,N} \bd{U}_{Nk} & = & \bd{0}   ,   \label{eq:set}  \\
\vdots && \nonumber   \\
\bd{S}_{N,k} + \hspace{3mm} \bd{S}_{N,k+1} \bd{U}_{k+1,k} + \hspace{0.5mm} \cdots + \hspace{3mm} \bd{S}_{NN} \bd{U}_{Nk} & = & \bd{0}   .   \nonumber
\eeq
In order to simplify our notation, we now define the block vectors $\bd{s}_{ij}$ (Eq.~\ref{eq:vecs})) and 
\beq
\bd{u}_{ij} = \left(\begin{array}{cccc} \bd{U}_{ij} & \bd{U}_{i+1,j} & \cdots & \bd{U}_{Nj} \end{array} \right)^T   ,
\eeq
which represent the portions of the $k$th columns of $\bd{S}$ and $\bd{U}$, respectively, which lie below, and inclusive of, the block with indices $(i,j)$.  We also let $\tilde{\bd{S}}_k$ represent the $k \times k$ block matrix formed from the lower-right corner of $\bd{S}$,
\beq
\tilde{\bd{S}}_k = \begin{bmatrix}
	\bd{S}_{N-k+1,N-k+1} & \bd{S}_{N-k+1,N-k+2} & \cdots & \bd{S}_{N-k+1,N} \\
	\bd{S}_{N-k+2,N-k+1} & \bd{S}_{N-k+2,N-k+2} & \cdots & \bd{S}_{N-k+2,N} \\
	\vdots & \vdots & \ddots & \vdots \\
	\bd{S}_{N,N-k+1} & \bd{S}_{N,N-k+2} & \cdots & \bd{S}_{N,N} \\
\end{bmatrix}   .
\eeq
With these definitions, we can rewrite Eqs. (\ref{eq:set}) in the matrix form:
\beq
\tilde{\bd{S}}_{N-k} \bd{u}_{k+1,k} = - \bd{s}_{k+1,k}   .
\eeq
Solving for the auxiliary block vector yields
\beq
\bd{u}_{k+1,k} = - \tilde{\bd{S}}^{-1}_{N-k} \bd{s}_{k+1,k}   .
\eeq
We now define $\bm{\sigma}^T_{ij}$ as the block row vector of $\bd{S}$ lying to the right, and inclusive of, the position $(i,j)$ (Eq. (\ref{eq:vecs})).  Inspecting Eq. (\ref{eq:SUNxN}), we can express the $k$th diagonal element of $\bd{SU}$ in the form
\beq
(\bd{SU})_{kk} & = & \bd{S}_{kk} + \bm{\sigma}^T_{k,k+1} \bd{u}_{k+1,k}   ,   \nonumber \\
	       & = & \bd{S}_{kk} - \bm{\sigma}^T_{k,k+1} \tilde{\bd{S}}^{-1}_{N-k} \bd{s}_{k+1,k}   ,   \nonumber \\
	       & = & \tilde{\bd{S}}_{N-k+1} / \bd{S}_{kk} , 
\eeq
where $\tilde{\bd{S}}_{N-k+1} / \bd{S}_{kk}$ denotes the Sch\"{u}r complement of $\tilde{\bd{S}}_{N-k+1}$ with respect to $\bd{S}_{kk}$.  Next, defining the matrices $\bm{\alpha}^{(k)}_{ij}$ as shown in Eq. (\ref{eq:alphas}), we write $(\bd{SU})_{kk} = \bm{\alpha}^{(N-k)}_{kk}$.  Since $\bd{SU}$ is upper triangular by design, the determinant is simply the product of the determinants of its diagonal blocks.  This, together with the fact that $\det(\bd{U}) = 1$ gives
\beq
\det(\bd{S}) = \prod^N_{k=1} \det (\bm{\alpha}^{(N-k)}_{kk})   .   \label{eq:det}
\eeq

In order to express $\det(\bd{S})$ in terms of the blocks of $\bd{S}$, we must now evaluate $\tilde{\bd{S}}^{-1}_k$, which appears in Eq. (\ref{eq:alphas}).  However, rather than approaching this problem directly, we will focus on the matrices $\bm{\alpha}^{(k)}_{ij}$, and seek to find a recursive relationship between matrices of consecutive values of $k$.  Thus, we begin by writing
\beq
\bm{\alpha}^{(k+1)}_{ij} & = & \bd{S}_{ij} - \bm{\alpha}^T_{i,N-k} \tilde{\bd{S}}^{-1}_{k+1} \bd{s}_{N-k,j}   \nonumber   \\
		       & = & \bd{S}_{ij} - \left[ \begin{smallmatrix} \bd{S}_{i,N-k} & \bm{\sigma}^T_{i,N-k+1} \end{smallmatrix} \right]
					\left[ \begin{smallmatrix} \bd{S}_{N-k,N-k} & \bm{\sigma}^T_{N-k,N-k+1} \\ \bd{s}_{N-k+1,N-k} & \tilde{\bd{S}}_{k} \end{smallmatrix} \right]^{-1}
				\left[\begin{smallmatrix} \bd{S}_{N-k,j} \\ \bd{s}_{N-k+1,j} \end{smallmatrix} \right]   ,   \nonumber  \\   \label{eq:intermediate}
\eeq
where we have partitioned the block vectors and matrix into sections of block length 1 and $N$.  Next, we evaluate the inverse matrix above by making use of the Banachiewic identity~\cite{Banachiewicz,Frazer}
\beq
\left[\begin{smallmatrix}
	\bd{A} & \bd{B} \\
	\bd{C} & \bd{D}
      \end{smallmatrix} \right]^{-1}
= \left[\begin{smallmatrix}
	 (\bd{A - BD}^{-1} \bd{C})^{-1} & - (\bd{A - BD}^{-1} \bd{C})^{-1} \bd{BD}^{-1} \\
	 -\bd{D}^{-1} \bd{C} (\bd{A - BD}^{-1} \bd{C})^{-1} & \bd{D}^{-1} \left[ \bd{I} + \bd{C} (\bd{A - BD}^{-1} \bd{C})^{-1} \bd{BD}^{-1} \right]
	\end{smallmatrix} \right]   .   \label{eq:Banachidentity}
\eeq
Identifying this expression with the partitioned form of $\tilde{\bd{S}}^{-1}_k$ in Eq. (\ref{eq:intermediate}), the Sch\"{u}r complement with respect to $\bd{S}_{N-k,N-k}$ becomes	
\beq
\bd{A - BD}^{-1} \bd{C} & = & \bd{S}_{N-k,N-k} - \bm{\sigma}^T_{N-k,N-k+1} \tilde{\bd{S}}^{-1}_{k} \bd{s}_{N-k+1,N-k}   \nonumber   \\
			& = & \bm{\alpha}^{(k)}_{N-k,N-k}   .
\eeq
Thus, evaluating the inverse matrix in Eq. (\ref{eq:intermediate}), we have
\beq
\bm{\alpha}^{(k+1)}_{ij} & = & \bd{S}_{ij} - \left[ \begin{smallmatrix} \bd{S}_{i,N-k} & \bm{\sigma}^T_{i,N-k+1} \end{smallmatrix} \right]   \nonumber \\
			 && \times  \left[ \begin{smallmatrix} (\bm{\alpha}^{(k)}_{N-k,N-k})^{-1} & - (\bm{\alpha}^{(k)}_{N-k,N-k})^{-1} \bm{\sigma}^T_{N-k,N-k+1} \tilde{\bd{S}}^{-1}_k \\ -\tilde{\bd{S}}^{-1}_k \bd{s}_{N-k+1,N-k} (\bm{\alpha}^{(k)}_{N-k,N-k})^{-1} & \tilde{\bd{S}}^{-1}_k \left[ \bd{I} + \bd{s}_{N-k+1,N-k} (\bm{\alpha}^{(k)}_{N-k,N-k})^{-1} \bm{\sigma}^T_{N-k,N-k+1} \tilde{\bd{S}}^{-1}_k \right] \end{smallmatrix} \right]   \nonumber   \\
			 && \times \left[\begin{smallmatrix} \bd{S}_{N-k,j} \\ \bd{s}_{N-k+1,j} \end{smallmatrix} \right]   ,   \nonumber   \\
			& = & \bd{S}_{ij} - \left[ \begin{smallmatrix} \left[ \bd{S}_{i,N-k} - \bm{\sigma}^T_{i,N-k+1} \tilde{\bd{S}}^{-1}_k \bd{s}_{N-k+1,N-k} \right] (\bm{\alpha}^{(k)}_{N-k,N-k})^{-1} \\ \bm{\sigma}^T_{i,N-k+1} \tilde{\bd{S}}^{-1}_k - \left[ \bd{S}_{i,N-k} - \bm{\sigma}^T_{i,N-k+1} \tilde{\bd{S}}^{-1}_k \bd{s}_{N-k+1,N-k}  \right] (\bm{\alpha}^{(k)}_{N-k,N-k})^{-1} \bm{\sigma}^T_{N-k,N-k+1} \tilde{\bd{S}}^{-1}_k  \end{smallmatrix} \right]   \nonumber   \\
			 && \hspace{10mm} \cdot \left[ \begin{smallmatrix} \bd{S}_{N-k,j} \\ \bd{s}_{N-k+1,j} \end{smallmatrix} \right]   ,   \nonumber   \\
			& = & \bd{S}_{ij} - \left[ \begin{smallmatrix} \bm{\alpha}^{(k)}_{i,N-k} (\bm{\alpha}^{(k)}_{N-k,N-k})^{-1} \\ \bm{\sigma}^T_{i,N-k+1} \tilde{\bd{S}}^{-1}_k - \bm{\alpha}^{(k)}_{i,N-k} (\bm{\alpha}^{(k)}_{N-k,N-k})^{-1} \bm{\sigma}^T_{N-k,N-k+1} \tilde{\bd{S}}^{-1}_k  \end{smallmatrix} \right] \cdot \left[ \begin{smallmatrix} \bd{S}_{N-k,j} \\ \bd{s}_{N-k+1,j} \end{smallmatrix} \right]   ,  \nonumber   \\
			& = & \bd{S}_{ij} - \bm{\alpha}^{(k)}_{i,N-k} (\bm{\alpha}^{(k)}_{N-k,N-k})^{-1} \bd{S}_{N-k,j} - \bm{\sigma}^T_{i,N-k+1} \tilde{\bd{S}}^{-1}_k \bd{s}_{N-k+1,j} \nonumber \\
			 && \hspace{10mm} + \bm{\alpha}^{(k)}_{i,N-k} (\bm{\alpha}^{(k)}_{N-k,N-k})^{-1} \bm{\sigma}^T_{N-k,N-k+1} \tilde{\bd{S}}^{-1}_k \bd{s}_{N-k+1,j}   ,   \nonumber \\
			& = & \left[ \bd{S}_{ij} - \bm{\sigma}^T_{i,N-k+1} \tilde{\bd{S}}^{-1}_k \bd{s}_{N-k+1,j} \right] \nonumber \\
			&& \hspace{10mm} - \bm{\alpha}^{(k)}_{i,N-k} (\bm{\alpha}^{(k)}_{N-k,N-k})^{-1} \left[ \bd{S}_{N-k,j} - \bm{\sigma}^T_{N-k,N-k+1} \tilde{\bd{S}}^{-1}_k \bd{s}_{N-k+1,j} \right]   ,   \nonumber \\
			& = & \bm{\alpha}^{(k)}_{ij} - \bm{\alpha}^{(k)}_{i,N-k} (\bm{\alpha}^{(k)}_{N-k,N-k})^{-1} \bm{\alpha}^{(k)}_{N-k,j}   .
\eeq
We can therefore express the $\bm{\alpha}^{(k)}_{ij}$ in the recursive form:
\beq
\bm{\alpha}^{(0)}_{ij} & = & \bd{S}_{ij}   ,   \nonumber \\
\bm{\alpha}^{(k+1)}_{ij} & = & \bm{\alpha}^{(k)}_{ij} - \bm{\alpha}^{(k)}_{i,N-k} (\bm{\alpha}^{(k)}_{N-k,N-k})^{-1} \bm{\alpha}^{(k)}_{N-k,j}   ,   \hspace{10mm} k \geq 1   .   \label{eq:alphas2}
\eeq
Given this recursive relationship, the matrix $\bm{\alpha}^{(0)}_{ij}$ can be read off directly from $\bd{S}$, and all higher order $\bm{\alpha}^{(k)}_{ij}$ can be calculated iteratively.  Finally, the determinant may be computed via Eq. (\ref{eq:det}).
\end{proof}

\section{Proof by Induction \label{sec:induction}}
Rather than calculating the determinant of $\bd{S}$ directly, as in the prior section, here we begin by establishing a recursive relationship between a sort of generalized Sch\"{u}r complement of $\bd{S}$.  To this end, we first prove the following lemma.

\begin{lemma}
Let $\bd{S}$ be a complex block matrix of the form
\beq
\bd{S} = \begin{bmatrix}
		\bd{S}_{11} & \bd{S}_{12} & \cdots & \bd{S}_{1N} \\
		\bd{S}_{21} & \bd{S}_{22} & \cdots & \bd{S}_{2N} \\
		\vdots  & \vdots  & \ddots & \vdots  \\
		\bd{S}_{N1} & \bd{S}_{N2} & \cdots & \bd{S}_{NN}
	 \end{bmatrix}   ,
\eeq
and let us define the set of block matrices $\left \{ \bm{\alpha}^{(0)}, \bm{\alpha}^{(1)}, \cdots, \bm{\alpha}^{(N-1)} \right \}$, where $\bm{\alpha}^{(k)}$ is an $(N - k) \times (N - k)$ block matrix with blocks
\beq
\bm{\alpha}^{(0)}_{ij} & = & \bd{S}_{ij}   \nonumber \\
\bm{\alpha}^{(k+1)}_{ij} & = & \bm{\alpha}^{(k)}_{ij} - \bm{\alpha}^{(k)}_{i,N-k} \left( \bm{\alpha}^{(k)}_{N-k,N-k} \right)^{-1} \bm{\alpha}^{(k)}_{N-k,j}   ,   \hspace{5mm} k \geq 1  .  \label{eq:alphas3}
\eeq
Then the determinants of consecutive $\bm{\alpha}^{(k)}$ are related via
\beq
\det(\bm{\alpha}^{(k)}) = \det(\bm{\alpha}^{(k+1)}) \det(\bm{\alpha}^{(k)}_{N-k,N-k})   .
\eeq
\end{lemma}

\begin{proof}
Let us partition $\bm{\alpha}^{(k)}$ in the form
\beq
\bm{\alpha}^{(k)} = \left[ \begin{array}{ccc|c}
		\bm{\alpha}^{(k)}_{11} & \cdots & \bm{\alpha}^{(k)}_{1,N-k-1} & \bm{\alpha}^{(k)}_{1,N-k} \\
		\vdots  & \ddots  & \vdots & \vdots  \\
		\bm{\alpha}^{(k)}_{N-k-1,1} & \cdots & \bm{\alpha}^{(k)}_{N-k-1,N-k-1} & \bm{\alpha}^{(k)}_{N-k-1,N-k} \\   \hline
		\bm{\alpha}^{(k)}_{N-k,1} & \cdots & \bm{\alpha}^{(k)}_{N-k,N-k-1} & \bm{\alpha}^{(k)}_{N-k,N-k}
	        \end{array} \right]   ,
\eeq
and let us denote the upper-left block of $\bm{\alpha}^{(k)}$ by $\tilde{\bm{\alpha}}^{(k)}$.  The Sch\"{u}r complement of $\bm{\alpha}^{(k)}$ with respect to $\tilde{\bm{\alpha}}^{(k)}$ is then
\beq
\bm{\alpha}^{(k)} /\tilde{\bm{\alpha}}^{(k)} & = & \begin{bmatrix}
					\bm{\alpha}^{(k)}_{11} & \cdots & \bm{\alpha}^{(k)}_{1,N-k-1} \\
					\vdots  & \ddots  & \vdots \\
					\bm{\alpha}^{(k)}_{N-k-1,1} & \cdots & \bm{\alpha}^{(k)}_{N-k-1,N-k-1}
					\end{bmatrix} \nonumber \\
		&& \hspace{10mm}     -	\begin{bmatrix} \bm{\alpha}^{(k)}_{N-k,1} \cdots \bm{\alpha}^{(k)}_{N-k,N-k-1} \end{bmatrix}
					\begin{bmatrix} \bm{\alpha}^{(k)}_{N-k,N-k} \end{bmatrix}^{-1}
					\begin{bmatrix} \bm{\alpha}^{(k)}_{1,N-k} \\ \vdots \\ \bm{\alpha}^{(k)}_{N-k-1,N-k} \end{bmatrix}   ,   \nonumber \\
		& = & \bigg[ \bm{\alpha}^{(k)}_{ij} - \bm{\alpha}^{(k)}_{i,N-k} \left( \bm{\alpha}_{N-k,N-k} \right)^{-1} \bm{\alpha}^{(k)}_{N-k,j} \bigg]   ,   \nonumber  \\
		& = & \bm{\alpha}^{(k+1)}   .
\eeq
Applying the $2 \times 2$ block identity (Eq. (\ref{eq:2x2det})), we can therefore write the determinant of $\bm{\alpha}^{(k)}$ as
\beq
\det( \bm{\alpha}^{(k)} ) = \det( \bm{\alpha}^{(k+1)} ) \det( \bm{\alpha}^{(k)}_{N-k,N-k} )   .   \label{eq:lemma}
\eeq
\end{proof}

\begin{theorem}
Given a complex block matrix of the form (\ref{eq:Smatrix}), and the matrices $\alpha^{(k)}_{ij}$ defined in Eq. (\ref{eq:alphas3}), the determinant of $\bd{S}$ is given by
\beq
\det ( \bd{S} ) = \prod^N_{k=1} \det( \bm{\alpha}^{(N-k)}_{kk} )   .
\eeq
\end{theorem}

\begin{proof}
Recall from the definition in Eq. (\ref{eq:alphas3}) that $\bm{\alpha}^{(0)} = \bd{S}$.  Starting with $\det(\bd{S})$ and applying Eq. (\ref{eq:lemma}) repeatedly yields
\beq
\det(\bd{S}) & = & \det( \bm{\alpha}^{(0)} )   ,   \nonumber \\
	     & = & \det (\bm{\alpha}^{(1)}) \det( \bm{\alpha}^{(0)}_{NN} )   ,   \nonumber \\
	     & = & \det (\bm{\alpha}^{(2)}) \det( \bm{\alpha}^{(1)}_{N-1,N-1} ) \det( \bm{\alpha}^{(0)}_{NN} )   ,   \nonumber \\
	     &&    \vdots \\
	     & = & \det (\bm{\alpha}^{(N-1)}) \det(\bm{\alpha}^{(N-2)}_{22}) \cdots \det( \bm{\alpha}^{(0)}_{NN} )   ,
\eeq
where the sequence of equalities terminates because $\bm{\alpha}^{(N-1)}$ is a $1 \times 1$ block matrix, which cannot be partitioned further.  Thus, writing $\bm{\alpha}^{(N-1)} = \bm{\alpha}^{(N-1)}_{11}$ we obtain the result
\beq
\det ( \bd{S} ) = \prod^N_{k=1} \det( \bm{\alpha}^{(N-k)}_{kk} )   .
\eeq

\end{proof}

\section{Applications \label{sec:examples}}
Having derived an expression for the determinant of an $N \times N$ complex block matrix, it will be useful to examine the result for a few specific values of $N$.  We choose to consider $N = 2$ and $N = 3$, as the first is the well-known result of Eq. (\ref{eq:2x2det}) and the second gives a clear picture of what sort of objects our result actually involves.  Larger values of $N$ quickly become cumbersome to write down, but the procedure for their calculation will be made clear.  Lastly, we present a ``real world" application of our result by calculating the determinant of a $48 \times 48$ matrix, which arises in the study of quark matter.

\subsection{$2 \times 2$ Block Matrices}
In the case $N = 2$, Eq. (\ref{eq:det}) reduces to
\beq
\det(\bd{S}) = \det(\bm{\alpha}^{(1)}_{11}) \det(\bm{\alpha}^{(0)}_{22})   ,   \label{eq:2x2det2}
\eeq
while Eqs. (\ref{eq:alphas2}) reduce to
\beq
\bm{\alpha}^{(0)}_{ij} = \bd{S}_{ij}   \hspace{10mm}   \bm{\alpha}^{(1)}_{ij} = \bd{S}_{ij} - \bd{S}_{i2} \bd{S}^{-1}_{22} \bd{S}_{21}   .
\eeq
Thus, we obtain the result
\beq
\det(\bd{S}) = \det(\bd{S}_{11} - \bd{S}_{12} \bd{S}^{-1}_{22} \bd{S}_{21}) \det(\bd{S}_{22})   .
\eeq
Note that this expression is identical to Eq. (\ref{eq:2x2block}).  It is instructive to combine the two determinants in this expression and rewrite it in the alternate forms
\beq
\det (\mathbf{S}) & = & \det( \bd{S}_{11} \bd{S}_{22} - \bd{S}_{12} \bd{S}^{-1}_{22} \bd{S}_{21} \bd{S}_{22} )   ,   \nonumber   \\
		  & = & \det( \bd{S}_{22} \bd{S}_{11} - \bd{S}_{22} \bd{S}_{12} \bd{S}^{-1}_{22} \bd{S}_{21} )   .
\eeq
In these forms, it is clear that when either $\bd{S}_{12}$ or $\bd{S}_{21}$ commute with $\bd{S}_{22}$ our expression reduces to:
\beq
&&  \bd{S}_{12} \bd{S}_{22} = \bd{S}_{22} \bd{S}_{12} \hspace{5mm} : \hspace{5mm}   \det(\bd{S}) = \det(\bd{S}_{22} \bd{S}_{11} - \bd{S}_{12} \bd{S}_{21})   ,   \nonumber \\
&&  \bd{S}_{21} \bd{S}_{22} = \bd{S}_{21} \bd{S}_{21} \hspace{5mm} : \hspace{5mm}   \det(\bd{S}) = \det(\bd{S}_{11} \bd{S}_{22} - \bd{S}_{12} \bd{S}_{21})   .   \label{eq:2x2result}
\eeq
Alternatively, for anti-commuting matrices, which often arise in the study of fermionic systems, the sign in (\ref{eq:2x2result}) becomes positive.

\subsection{$3 \times 3$ Block Matrices}
In the case $N = 3$, Eq. (\ref{eq:det}) reduces to
\beq
\det(\bd{S}) = \det(\bm{\alpha}^{(2)}_{11}) \det(\bm{\alpha}^{(1)}_{22}) \det(\bm{\alpha}^{(0)}_{33})   ,
\eeq
while Eqs. (\ref{eq:alphas2}) reduce to
\beq
& \bm{\alpha}^{(0)}_{ij} = \bd{S}_{ij} ,  \hspace{10mm}   \bm{\alpha}^{(1)}_{ij} = \bd{S}_{ij} - \bd{S}_{i3} \bd{S}^{-1}_{33} \bd{S}_{31} ,  &   \nonumber   \\
& \bm{\alpha}^{(2)}_{ij} = \left[ \bd{S}_{ij} - \bd{S}_{i3} \bd{S}^{-1}_{33} \bd{S}_{3j} \right] \hspace{90mm} & \nonumber \\
&- \left[ \bd{S}_{i2} - \bd{S}_{i3} \bd{S}^{-1}_{33} \bd{S}_{32} \right] \left[ \bd{S}_{22} - \bd{S}_{23} \bd{S}^{-1}_{33} \bd{S}_{32} \right]^{-1} \left[ \bd{S}_{2j} - \bd{S}_{23} \bd{S}^{-1}_{33} \bd{S}_{3j} \right] . &   \nonumber \\
\eeq
Thus, we obtain the result
\beq
\det(\bd{S}) & = & \det \bigg( \left[ \bd{S}_{11} - \bd{S}_{13} \bd{S}^{-1}_{33} \bd{S}_{31} \right] \nonumber \\
	     && \hspace{5mm} - \left[ \bd{S}_{12} - \bd{S}_{13} \bd{S}^{-1}_{33} \bd{S}_{32} \right] \left[ \bd{S}_{22} - \bd{S}_{23} \bd{S}^{-1}_{33} \bd{S}_{32} \right]^{-1} \left[ \bd{S}_{21} - \bd{S}_{23} \bd{S}^{-1}_{33} \bd{S}_{31} \right] \bigg)   \nonumber   \\
	     && \hspace{5mm} \times \det \left( \bd{S}_{22} - \bd{S}_{23} \bd{S}^{-1}_{33} \bd{S}_{32} \right) \det(\bd{S}_{33})   .
\eeq
Analogously to the $N = 2$ case, the commutation of certain blocks (e.g., $\bd{S}_{33}$ and $\bd{S}_{3j}$, $\alpha^{(1)}_{12}$ with $\alpha^{(1)}_{22}$) allows this expression to be simplified.

\subsection{Eigenenergies of high-density quark matter}

Having considered the general form of the determinant of $2 \times 2$ and $3 \times 3$ block matrices, we now consider a true application of our result by calculating the eigenenergies of quark matter in the two flavor Nambu--Jona-Lasinio model~\cite{Rossner,Klevansky,Buballa}.  In this model, the energies are the roots of the equation $\det{\bd{S}} = 0$, where
\beq
\bd{S} = \begin{bmatrix}
		\slashed{k} + \mu \gamma^0 - M & \Delta \gamma_5 \tau_2 \lambda_2 \\
		- \Delta^\ast \gamma_5 \tau_2 \lambda_2 & \slashed{k} - \mu \gamma^0 - M   \label{eq:NJL}
	 \end{bmatrix}   ,
\eeq
and where $M$ is the effective quark mass, $\mu$ is the chemical potential, $\Delta$ is the pairing gap, $\slashed{k} \equiv E \gamma^0 - \bm{\gamma} \cdot \mathbf{k}$, where $\bd{k} = (k_x, k_y, k_z)^T$ is the quark momentum and the $\gamma^\nu$ ($\nu = 0...3$) are the $4 \times 4$ Dirac matrices:
\beq
\gamma^0 = \begin{bmatrix} \bd{I} & \bd{0} \\ \bd{0} & -\bd{I} \end{bmatrix}   ,   \hspace{10mm}
\gamma^i = \begin{bmatrix} \bd{0} & \bm{\sigma}_i \\ -\bm{\sigma}_i & \bd{0} \end{bmatrix}   ,
\eeq
with $\bm{\sigma} = (\sigma_x, \sigma_y, \sigma_z)^T$ representing the vector of Pauli matrices:
\beq
\sigma_x = \begin{bmatrix} 0 & 1 \\ 1 & 0 \end{bmatrix}   ,   \hspace{10mm}
\sigma_y = \begin{bmatrix} 0 & -i \\ i & 0 \end{bmatrix}  ,   \hspace{10mm}
\sigma_z = \begin{bmatrix} 1 & 0 \\ 0 & -1 \end{bmatrix}  .
\eeq
In addition, $\gamma_5 \equiv i \gamma^0 \gamma^1 \gamma^2 \gamma^3$, $\tau_2$ is the second Pauli matrix in flavor space, and $\lambda_2$ is the second Gell-Mann matrix in color space:
\beq
\lambda_2 = \begin{bmatrix}
		0 & -i & 0 \\
		i & 0 & 0 \\
		0 & 0 & 0
	    \end{bmatrix}   .
\eeq
Thus, while we have written Eq. (\ref{eq:NJL}) in $2 \times 2$ block form, each block is itself a $24 \times 24$ matrix (2 (flavor) $\times$ 3 (color) $\times$ 4 (Dirac)).

Before constructing the determinant from Eq. (\ref{eq:det}), we must choose how we wish to partition $\bd{S}$.  We could, for instance, choose to begin with the $2 \times 2$ block form shown in Eq. (\ref{eq:NJL}), in which case the $\bm{\alpha}^{(k)}$ would be $24 \times 24$ matrices, which we would partition further, repeating the process until we have eliminated all indices.  While this choice has the advantage of requiring the construction of only a single $\bm{\alpha}^{(k)}$ ($\bm{\alpha}^{(1)}$) for the first step (Eq. (\ref{eq:2x2det2})), the price is that correspondingly more steps are required to finally obtain $\det{\bd{S}}$.

As a middle ground, balancing the number of $\bm{\alpha}^{(k)}$'s which must be constructed in each step with the number of steps, we will partition $\bd{S}$ into a $6 \times 6$ block matrix, with each block of size $8 \times 8$.  We achieve this partitioning by writing out the color indices explicitly, while leaving the Dirac and flavor indices intact.  Thus, the non-zero blocks of $\bd{S}$ become
\beq
\bd{S}_{11} & = & \bd{S}_{22} = \bd{S}_{33} = \slashed{k} + \mu \gamma^0 - M   ,   \nonumber   \\
\bd{S}_{44} & = & \bd{S}_{55} = \bd{S}_{66} = \slashed{k} - \mu \gamma^0 - M  ,   \nonumber   \\
\bd{S}_{24} & = & - \bd{S}_{15} = i \Delta \gamma_5 \tau_2    ,   \\
\bd{S}_{42} & = & - \bd{S}_{51} = i \Delta^\ast \gamma_5 \tau_2    .   \nonumber
\eeq

We now must construct $\bm{\alpha}^{(1)} \cdots \bm{\alpha}^{(5)}$.  From Eq. (\ref{eq:alphas2}), we see that $\alpha^{(k+1)}_{ij}$ will be equal to $\alpha^{(k)}_{ij}$ unless both $\bm{\alpha}^{(k)}_{i,N-k}$ and $\bm{\alpha}^{(k)}_{N-k,j}$ are non-zero.  As a result, since $\bd{S}_{i6} = \bd{S}_{6j} = \bd{0}$ for $i,j \neq 6$, we find
\beq
\bm{\alpha}^{(1)}_{ij} = \bd{S}_{ij}   ,   \hspace{5mm}   1 \leq i,j \leq 5   .
\eeq
Next, we note that $\bd{S}_{15}$ and $\bd{S}_{51}$ are the only non-zero blocks with a row or column index of 5.  Thus, the only $\bm{\alpha}^{(2)}_{ij}$ which differs from $\bm{\alpha}^{(1)}_{ij}$ is
\beq
\bm{\alpha}^{(2)}_{11} = \bd{S}_{11} - \bd{S}_{15} \bd{S}^{-1}_{55} \bd{S}_{51}   .   \label{eq:alpha211}
\eeq
A straightforward application of the Banachiewicz identity (Eq. (\ref{eq:Banachidentity})) yields $\bd{S}^{-1}_{55}$:
\beq
\bd{S}^{-1}_{55} = \frac{\slashed{k} - \mu \gamma^0 + M}{(E - \mu)^2 - E^2_k}   ,
\eeq
where $E_k \equiv \sqrt{\bd{k}^2 + M^2}$.  Substituting this expression into Eq. (\ref{eq:alpha211}) yields
\beq
\bm{\alpha}^{(2)}_{11} & = & \slashed{k} + \mu \gamma^0 - M - (-i \Delta \gamma_5 \tau_2) \left[ \frac{\slashed{k} - \mu \gamma^0 + M}{(E - \mu)^2 - E^2_k} \right] (-i \Delta \gamma_5 \tau_2)   ,   \nonumber   \\
		       & = & \slashed{k} + \mu \gamma^0 - M -  |\Delta|^2 \hspace{.5mm} \frac{\slashed{k} - \mu \gamma^0 - M}{(E - \mu)^2 - E^2_k}   ,   \label{eq:taugone}
\eeq
where we have used the fact that $\tau^2_2 = \bd{I}$ and $\gamma_5 \gamma^\mu = -\gamma^\mu \gamma_5$.  Meanwhile, the rest of the $\bm{\alpha}^{(2)}_{ij} = \bm{\alpha}^{(1)}_{ij} = \bd{S}_{ij}$, so that
\beq
\bm{\alpha}^{(2)}_{ij} = \begin{cases}
				\slashed{k} + \mu \gamma^0 - M -  |\Delta|^2 \hspace{.5mm} \frac{\slashed{k} - \mu \gamma^0 - M}{(E - \mu)^2 - E^2_k}   & \mbox{if } i,j = 1,  \\
				\bd{S}_{ij} & \mbox{else } (1 \leq i,j \leq 4)  .
			 \end{cases}
\eeq
Constructing $\bm{\alpha}^{(3)}$, we note that $\bm{\alpha}^{(2)}_{24}$ and $\bd{S}_{42}$ are the only non-zero blocks in the fourth row or column of $\bm{\alpha}^{(2)}$.  As a result, we find
\beq
\bm{\alpha}^{(3)}_{22} & = & \bm{\alpha}^{(2)}_{22} - \bm{\alpha}^{(2)}_{24} \left( \bm{\alpha}^{(2)}_{44} \right)^{-1} \bm{\alpha}^{(2)}_{42}   ,   \nonumber   \\
		       & = & \bd{S}_{11} - (-\bd{S}_{15}) (\bd{S}_{55})^{-1} (-\bd{S}_{51})   ,   \nonumber   \\
		       & = & \bd{S}_{11} - \bd{S}_{15} \bd{S}^{-1}_{55} \bd{S}_{51}   ,   \nonumber   \\
		       & = & \bm{\alpha}^{(2)}_{11}   .
\eeq
The rest of the $\bm{\alpha}^{(3)}_{ij}$ are equal to $\bm{\alpha}^{(2)}_{ij}$ so we find
\beq
\bm{\alpha}^{(3)}_{ij} = \begin{cases}
				\slashed{k} + \mu \gamma^0 - M -  |\Delta|^2 \hspace{.5mm} \frac{\slashed{k} - \mu \gamma^0 - M}{(E - \mu)^2 - E^2_k}   & \mbox{if } i,j = 1 \mbox{ or } i,j = 2,  \\
				\bd{S}_{ij} & \mbox{else } (1 \leq i,j \leq 3)  .
			 \end{cases}
\eeq
Since $\bm{\alpha}^{(3)}_{33} = \bd{S}_{33}$ is the only non-zero block in the third row/column, we have
\beq
\bm{\alpha}^{(4)}_{ij} = \bm{\alpha}^{(3)}_{ij}   ,   \hspace{5mm}   1 \leq i,j \leq 2   .
\eeq
Finally, since $\bm{\alpha}^{(4)}_{12} = \bm{\alpha}^{(4)}_{21} = \bd{S}_{12} = \bd{S}_{21} = \bd{0}$, we find that the sole block of $\bm{\alpha}^{(5)}$ is
\beq
\bm{\alpha}^{(5)}_{11} & = & \bm{\alpha}^{(4)}_{11} = \bm{\alpha}^{(3)}_{11}   ,   \nonumber \\
		       & = & \slashed{k} + \mu \gamma^0 - M -  |\Delta|^2 \hspace{.5mm} \frac{\slashed{k} - \mu \gamma^0 - M}{(E - \mu)^2 - E^2_k}   .
\eeq
Having constructed the necessary determinants, Eq. (\ref{eq:det}) gives
\beq
\det(\bd{S}) & = & \det ( \bm{\alpha}^{(5)}_{11} ) \det( \bm{\alpha}^{(4)}_{22} ) \det( \bm{\alpha}^{(3)}_{33} ) \det( \bm{\alpha}^{(2)}_{44} ) \det( \bm{\alpha}^{(1)}_{55} ) \det( \bm{\alpha}^{(0)}_{66} )   ,   \nonumber \\
	     & = & \det_{f,D} \left( \slashed{k} + \mu \gamma^0 - M - |\Delta|^2 \hspace{.5mm} \frac{\slashed{k} - \mu \gamma^0 - M}{(E - \mu)^2 - E^2_k} \right)^2 \det_{f,D} ( \slashed{k} + \mu \gamma^0 - M)   \nonumber \\
	     && \hspace{10mm} \times  \det_{f,D} ( \slashed{k} - \mu \gamma^0 - M)^3   ,
\eeq
where the subscript of $f,D$ on the determinant indicates that the remaining determinant is to be taken over flavor and Dirac indices.

Thus, we have eliminated 6 of the original 48 indices, and have only the Dirac and flavor indices remaining.  In fact, the flavor indices are now trivial, having vanished upon squaring the matrix $\tau_2$ (see Eq. (\ref{eq:taugone})), so that
\beq
\det(\bd{S}) & = & \det_{D} \left( \slashed{k} + \mu \gamma^0 - M - |\Delta|^2 \hspace{.5mm} \frac{\slashed{k} - \mu \gamma^0 - M}{(E - \mu)^2 - E^2_k} \right)^4 \det_{D} ( \slashed{k} + \mu \gamma^0 - M)^2   \nonumber   \\
	     && \hspace{10mm} \times \det_{D} ( \slashed{k} - \mu \gamma^0 - M)^6   .   \label{eq:almost}
\eeq
Computing $\det (\slashed{k} \pm \mu \gamma^0 - M)$ yields
\beq
\det (\slashed{k} \pm \mu \gamma^0 - M) & = & \det \begin{bmatrix} (E \pm \mu) - M & - \bm{\sigma} \cdot \bd{k} \\
								\bm{\sigma} \cdot \bd{k} & - (E \pm \mu) - M
						\end{bmatrix}   ,   \nonumber   \\
				    & = & \left[ - (E \pm \mu)^2 + M^2 + (\sigma \cdot \bd{k})^2 \right]^2   ,   \nonumber   \\
				    & = & \left[ (E \pm \mu)^2 - E^2_k \right]^2  ,   \nonumber \\
				    & = & \left[ E + (E_k \pm \mu) \right]^2 \left[E - (E_k \mp \mu) \right]^2   ,   \label{eq:part1}
\eeq
where we have used the fact that $(\bm{\sigma} \cdot \bd{k})^2 = \bd{k}^2$.

Finally, computing the remaining determinant from Eq. (\ref{eq:almost}), we find
\beq
&& \det \left( \slashed{k} + \mu \gamma^0 - M -  |\Delta|^2 \hspace{.5mm} \frac{\slashed{k} - \mu \gamma^0 - M}{(E - \mu)^2 - E^2_k} \right)   \nonumber \\
&& \hspace{10mm} = \left \{ - \left[(E + \mu) - |\Delta|^2 \frac{E - \mu}{(E - \mu)^2 - E^2_k} \right]^2 + \left[ 1 - \frac{|\Delta|^2}{(E - \mu)^2 - E^2_k} \right]^2 E^2_k \right \}^2   ,   \nonumber   \\
&& \hspace{10mm} = \frac{ [ E^2 - (E_k + \mu)^2 - |\Delta|^2 ]^2 [ E^2 - (E_k - \mu)^2 - |\Delta|^2 ]^2}{(E - E_k - \mu)^2 (E + E_k - \mu)^2}   .   \label{eq:part2}
\eeq
Inserting Eq. (\ref{eq:part1}) and (\ref{eq:part2}) into Eq. (\ref{eq:almost}), we obtain the result
\beq
\det(\bd{S}) & = & \left[ E + \sqrt{ (E_k + \mu)^2 + |\Delta|^2 } \right]^8 \left[ E + \sqrt{ (E_k - \mu)^2 + |\Delta|^2 } \right]^8   \nonumber   \\
	     && \hspace{10mm} \times \left[ E - \sqrt{ (E_k + \mu)^2 + |\Delta|^2 } \right]^8 \left[ E - \sqrt{ (E_k - \mu)^2 + |\Delta|^2 } \right]^8   \nonumber   \\
	     && \hspace{10mm} \times (E + E_k + \mu)^4 (E - E_k - \mu)^4 (E + E_k - \mu)^4 (E - E_k + \mu)^4   .   \nonumber \\
\eeq
Finally, then, we can read off the eigenenergies, which are the absolute values of the roots of $\det(\bd{S}) = 0$:
\beq
E_1 & = & | E_k + \mu |      \hspace{30mm} (\mbox{multiplicity } 8) , \nonumber   \\
E_2 & = & | E_k - \mu |     \hspace{30mm} (\mbox{multiplicity } 8) , \nonumber   \\
E_3 & = & \sqrt{ (E_k + \mu)^2 + |\Delta|^2 }   \hspace{11mm} (\mbox{multiplicity } 16)   ,   \\
E_4 & = & \sqrt{ (E_k - \mu)^2 + |\Delta|^2 }   \hspace{11mm} (\mbox{multiplicity } 16)   .   \nonumber
\eeq
Indeed, these are the correct eigenenergies, as reported previously by Rossner~\cite{Rossner}.

\section{Acknowledgements}
This research was supported in part by NSF grant PHY09-69790.  The author would like to thank Gordon Baym for his guidance throughout the research which inspired this work.

\end{document}